\documentclass{amsart}
\usepackage{amsmath, amssymb, amsthm}
\usepackage{mathtools}
\usepackage{amscd}
\usepackage{graphicx} 
\usepackage{enumerate}
\usepackage{todonotes}
\usepackage{tabularx}
\usepackage{nicefrac}

\def\CC{{\mathbb C}}

\def\QQ{{\mathbb Q}}
\def\PP{{\mathbb P}}
\def\QQ{{\mathbb Q}}
\def\RR{{\mathbb R}}

\def\Qbar{\overline{\mathbb Q}}

\def\0{{\mathbf 0}}
\def\1{{\mathbf 1}}

\def\GL{\mathrm{GL}}

\def\min{\mathrm{min}}

\def\Do{\mathrm{Do}}

\title{Quantitative height bounds under splitting conditions}
 
\author[Fili]{Paul A. Fili}
\address{Department of Mathematics\\ Oklahoma State University, Stillwater, OK 74078}
\email{fili@post.harvard.edu}
\author[Pottmeyer]{Lukas Pottmeyer}
\address{Fakultät f\"ur Mathematik\\ Universit\"at Duisburg-Essen, D-45117 Essen}
\email{lukas.pottmeyer@uni-due.de}

\date{\today}

\usepackage[pdftitle={Height bounds under splitting conditions},pdfauthor={Fili and Pottmeyer},pdfstartview={}]{hyperref}

\usepackage[all]{xy}

\newtheorem{thm}{Theorem}

\newtheorem{prop}[thm]{Proposition}
\newtheorem{lemma}[thm]{Lemma}

\newtheorem*{thm*}{Theorem}
\newtheorem*{alg*}{Algorithm}
\newtheorem*{lemma*}{Lemma}

\theoremstyle{remark}

\newtheorem*{rmk*}{Remark}

\newtheorem*{notation*}{Notation}
\newtheorem{example}[thm]{Example}
\newtheorem*{example*}{Example}

\theoremstyle{definition}

\newtheorem*{defn*}{Definition}

\newcommand{\mybf}{\mathbb}

\newcommand{\bP}{\mybf{P}}
\newcommand{\bR}{\mybf{R}}

\newcommand{\bC}{\mybf{C}}
\newcommand{\bN}{\mybf{N}}

\newcommand{\bQ}{\mybf{Q}}

\newcommand{\bA}{\mybf{A}}

\newcommand{\cO}{\mathcal{O}}

\newcommand{\al}{\alpha}

\providecommand{\abs}[1]{\lvert#1\rvert}

\newcommand{\ON}[1]{\operatorname{#1}}

\newcommand{\ra}{\rightarrow}

\newcommand{\ep}{\epsilon}

\newcommand{\p}{\partial}

\newcommand{\Diag}{\mathrm{Diag}}

\def\talltareesidedbox#1{\setbox0=\hbox{$#1$}\dimen0=\wd0 \advance\dimen0 by3pt\rlap{\hbox{\vrule height10pt width.4pt
 depth2pt \kern-.4pt\vrule height10.4pt width\dimen0 depth-10pt\kern-.4pt \vrule height10pt width.4pt depth2pt}}
 \relax \hbox to\dimen0{\hss$#1$\hss}}
\def\tareesidedbox#1{\setbox0=\hbox{$#1$}\dimen0=\wd0 \advance\dimen0 by3pt\rlap{\hbox{\vrule height8pt width.4pt
 depth2pt \kern-.4pt\vrule height8.4pt width\dimen0 depth-8pt\kern-.4pt \vrule height8pt width.4pt depth2pt}}
\relax \hbox to\dimen0{\hss$#1$\hss}}
\def\shorttareesidedbox#1{\setbox0=\hbox{$#1$}\dimen0=\wd0 \advance\dimen0 by3pt\rlap{\hbox{\vrule height7pt width.4pt
 depth2pt \kern-.4pt\vrule height7.4pt width\dimen0 depth-7pt\kern-.4pt \vrule height7pt width.4pt depth2pt}}
 \relax \hbox to\dimen0{\hss$#1$\hss}}

\newcommand{\sP}{\mathsf{P}}
\newcommand{\sA}{\mathsf{A}}

\begin{document}

 \begin{abstract}
In an earlier work, the first author and Petsche used potential theoretic techniques to establish a lower bound for the height of algebraic numbers that satisfy splitting conditions, such as being totally real or $p$-adic, improving on earlier work of Bombieri and Zannier in the totally $p$-adic case. These bounds applied as the degree of the algebraic number over the rationals tended towards infinity. In this paper, we use discrete energy approximation techniques on the Berkovich projective line to make the dependence on the degree in these bounds explicit, and we establish lower bounds for algebraic numbers which depend only on local properties of the numbers. 
 \end{abstract}

\thanks{The authors would like to thank the Erwin Schr\"odinger International Institute for Mathematical Physics at the University of Vienna for their support during the `Heights and Diophantine geometry, group theory and additive combinatorics' workshop during which some of this research took place. The authors would also like to thank Igor Pritsker for numerous helpful conversations, particularly regarding the proof of Theorem \ref{thm:robin-constant}. Lastly, the authors would like to thank the anonymous referee for a very careful reading of the manuscript with numerous helpful comments, including suggesting improvements to Theorems \ref{thm:robin-constant-for-O-L}, \ref{thm:robin-constant-for-O-L-cross}, and \ref{thm:robin-constant-for-L-p}. The second author was supported by the DFG-Projekt \textit{Heights and unlikely intersections} HA~6828/1-1.}

\maketitle

\section{Introduction}
In a previous work of the first author and Petsche \cite{F-P-EIOLF}, it was established that if $S$ is a set of rational places and $L_S$ denotes an extension of $\bQ$ containing all algebraic numbers whose Galois conjugates all lie in the local fields $L_p\neq \CC$ for each $p\in S$, then for all $\al\in L_S$,
\begin{equation}\label{SimpleGlobalThmBound}
h(\alpha)\geq \frac{1}{2} \sum_{p\in S} I(\mu_{L_p}) + o(1)\quad \text{as}\quad d=[\bQ(\al):\bQ]\ra \infty.
\end{equation}
The local energies $I(\mu_{L_p})$ appearing above arise naturally as the solutions to a certain energy minimization problem for local fields, and the resulting bounds improved on earlier constants obtained by Bombieri and Zannier \cite{BombieriZannierNote}.

The goal of this note is to establish a bound in which the dependence on the degree is made explicit. We begin by fixing some notation to be used throughout this paper:
\begin{center}
\begin{tabularx}{\linewidth}{rX}
 $S\subseteq M_\bQ$ & will be a given set of rational primes, possibly containing the archimedean prime.\\
 $L_p/\bQ_p$ & will be a given finite normal extension for each $p\in S$.\\
 $L_S$ & will denote the field of all algebraic numbers all of whose Galois conjugates lie in $L_p$ for each $p\in S$.
\end{tabularx}
\end{center}
\noindent Further, for each finite prime $p\in S$, we will denote by:
\begin{center}
\begin{tabular}{rl}
 $e=e_p$ & the ramification degree of $L_p/\bQ_p$,\\
 $f=f_p$ & the inertial degree of $L_p/\bQ_p$,\\
 $q = p^f$ & the order of the residue field of $L_p$, and\\
 $O_{L_p}$ & the ring of integers of $L_p$.
\end{tabular}
\end{center}
\noindent To ease notation and avoid too many subscripts, we will leave tacit the dependence of $e,f,q$ on the prime $p\in S$ below. Our first result is the following:
\begin{thm}\label{thm:1}
 Let $L_S$ be as above and $\al\in L_S$ with $d=[\bQ(\al):\bQ]>1$. Set 
\[
 V_{\infty} = \begin{dcases} 0 & \text{ if } \infty \notin S \\
               \max\left\{\frac{7\zeta(3)}{4\pi^2} - \frac{0.95 d + 2}{2 d^2} - \frac{(d-2) \log{d}}{2 d (d-1)},0\right\} & \text{ if } \infty \in S
              \end{dcases}
\]
where $\zeta(3) = \sum_{n\geq 1} \nicefrac{1}{n^3}$. 
Then we have
\begin{equation}\label{eqn:bound}
 h(\al) \geq - \frac{\log d}{2(d-1)} + V_{\infty} + \frac{1}{2} \sum_{\substack{p\in S\\ p\neq \infty ,\ p^{1/e} < d}} \left((1 - \frac{1}{q^{n_p}}) \frac{q\log{p}}{e(q^2 -1)} - \frac{\log d}{d} \right)
\end{equation}
where for each $p$,
\[
 n_p = \left\lfloor \frac{e \log d}{\log p}\right\rfloor .
\]
If in addition $\al$ is an algebraic integer, then we have
\begin{equation}\label{eqn:bound-integer}
 h(\al) \geq - \frac{\log d}{2(d-1)} + V_{\infty} + \frac{1}{2} \sum_{\substack{p\in S\\ p \neq \infty ,\ p^{1/e} < d}} \left((1 - \frac{1}{q^{n_p}}) \frac{\log{p}}{e(q -1)} - \frac{\log d}{d} \right),
\end{equation}
and if $\al$ is an algebraic unit, then 
\begin{equation}\label{eqn:bound-unit}
 h(\al) \geq - \frac{\log d}{2(d-1)} + V_{\infty} + \frac{1}{2} \sum_{\substack{p\in S\\ p \neq \infty ,\ p^{1/e} < d}} \left((1 - \frac{1}{q^{n_p}}) \frac{q\log{p}}{e(q -1)^2} - \frac{\log d}{d} \right).
\end{equation}
\end{thm}
\noindent Theorem \ref{thm:1} should be compared to Bombieri and Zannier \cite[Theorem 3]{BombieriZannierNote}. Note that the result of Bombieri and Zannier does not cover the case where $\infty\in S$, but uses a form of Mahler's inequality much as our result does in \eqref{eqn:bound}. We note that in case $\infty \in S$, $V_{\infty}$ is positive if and only if $d>6$, i.e. only for algebraic numbers of degree $> 6$.

 With trivial modifications, Theorem \ref{thm:1} can be stated over an arbitrary base number field $K/\bQ$. This statement is given as Theorem \ref{thm:general-thm-1} below.

Note that the height bounds from Theorem \ref{thm:1} are negative for algebraic numbers of small degree $d$. But as there are only finitely many algebraic numbers of bounded height and bounded degree, we recover the aforementioned result of the first author and Petsche:
\[
\liminf_{\al \in L_S} h(\al) \geq \sum_{p\in S\setminus\{\infty\}} \frac{q\log{p}}{e(q^2 -1)} + \begin{cases} \frac{7\zeta(3)}{4\pi^2} &\text{ if } \infty \in S \\ 0 &\text{ else } \end{cases}.
\]
Although our bounds are trivial in some cases, we can use Theorem \ref{thm:1} to give absolute lower bounds for non roots of unity in $L_S^\times$. As such lower bounds depend mainly on the smallest prime in $S$, we will focus on the case $\vert S \vert =1$. Note that we have trivially $L_S \subseteq L_{S'}$, whenever $S' \subseteq S$. In case $S=\{\infty\}$, Schinzel \cite{SchinzelTotReal} gave the sharp lower bound $\nicefrac{1}{2} \log\,(1+\sqrt{5}) /2$.

\begin{thm}\label{thm:2}
 Let $p$ be a rational prime and $\al\in L_{\{p\}}^\times$, not a root of unity. Then we have
\[
 h(\al)\geq \begin{dcases} \frac{\log{p}}{13(q-1) e \log\left(\frac{5(q-1)e}{\log{p}}\right)^4} & \text{ if } e \geq 2 \\
 \frac{\log(2) \log(p)}{5 (q+1) \log\left(\frac{5(q+1)}{\log p}\right)} & \text{ if } e=1 \end{dcases} 
\]
\end{thm}

\subsection{Background}\label{sec:background}
We will now recall some of the notation and results regarding potential theory and Berkovich space which we will use. For background we refer the reader to \cite{BakerRumelyBook,FRL,FRLcorrigendum}. Our notation largely follows that of Favre and Rivera-Letelier \cite{FRL,FRLcorrigendum}. For simplicity, we will state our results with $\bQ$ as the ground field, however, any number field $K$ can be substituted for the ground field with the usual renormalizations of absolute values.

We will denote by $\bA^1,\bP^1$ the usual affine and projective lines and by $\sA^1,\sP^1$ the Berkovich affine and projective lines, respectively. We refer the reader to \cite{BakerRumelyBook,FRL,BakerBerkArticle} for some basic references on Berkovich space. We define the \emph{standard measures} $\lambda_p$ on $\sP^1(\bC_p)$ to be the probability measures which are either the Dirac measure on the Gauss point of $\sP^1(\bC_p)$ if $p\nmid\infty$ or the normalized Haar measure on the unit circle of $\bC^\times$ if $p\mid\infty$. We let $\Delta$ denote the measure-valued Laplacian on $\sP^1$. We call $\rho=(\rho_p)_{p\in M_\bQ}$ an \emph{adelic measure} if for each $p\in M_\bQ$, $\rho_p$ is a Borel probability measure on $\sP^1(\bC_p)$ which is equal to $\lambda_p$ for all but finitely many $p$ and \emph{admits a continuous potential} with respect to $\lambda_p$ at the remaining places in the sense of \cite{FRL}, that is to say, for which
 $
  \rho_p-\lambda_p = \Delta g
 $
 for some $g\in C(\sP^1(\bC_p))$. For any $\rho_p,\sigma_p$ signed finite Borel measures on $\sP^1(\bC_p)$, we define, when it exists, the \emph{local mutual energy pairing} to be
\begin{equation}\label{eqn:FRL-pairing}
 (\rho_p,\sigma_p)_p = {\iint}_{\sA^1_p\times\sA^1_p\setminus\Diag_p} -\log \abs{x-y}_p\,d\rho_p(x)\,d\sigma_p(y)
\end{equation}
where $\Diag_p = \{ (x,x) : x\in\bC_p\}$ denotes the diagonal of classical (or `type I') points and $\abs{\cdot}_p$ denotes the usual $p$-adic (or archimedean if $p\mid \infty$) absolute value, normalized so as to agree with the usual Euclidean absolute value when $p=\infty$ and for finite primes to satisfy the usual normalization for the $p$-adic absolute value where  $\abs{p}_p=1/p$. Note that the notation in our integral here is loose in the $p$-adic setting, where for non-classical points $x$ or $y$, the distance $\abs{x-y}_p$ should be read as the natural extension of $\abs{x-y}_p$ to the Berkovich projective line, denoted by $\sup\{x,y\}$ in the article of Favre and Rivera-Letelier \cite[\S 3.3]{FRL} and as the \emph{Hsia kernel} $\delta(x,y)_\infty$ in the book of Baker and Rumely \cite[\S 4]{BakerRumelyBook}.

When $\rho=(\rho_p),\sigma=(\sigma_p)$ are adelic measures we will sometimes write $(\rho,\sigma)_p$ instead of $(\rho_p,\sigma_p)_p$ to ease notation. When well-defined it is easy to see that the local mutual energy is symmetric. The local mutual energy exists in particular when $\rho_p$ and $\sigma_p$ are either Borel probability measures with continuous potentials with respect to the standard measure or are probability measures supported on a finite subset of $\bP^1(\overline \bQ)$. In particular this applies for our adelic measures, and extends naturally by bilinearity to the vector space of signed measures arising from these measures. We refer the reader to \cite{FRL} for proofs of these results.

For a local field $L_p/\bQ_p$ with absolute value $\abs{\cdot}=\abs{\cdot}_p$, the first author and Petsche \cite{F-P-EIOLF} defined the \emph{energy integral} of a Borel probability measure $\nu$ on $L_p$ to be 
\begin{equation*}\label{eqn:EnergyIntegral}
I(\nu)=\iint_{\PP^1(L_p)\times\PP^1(L_p)}-\log\delta(x,y)\,d\nu(x)\,d\nu(y),
\end{equation*}
where $\delta:\PP^1(L_p)\times\PP^1(L_p)\to\RR$ is defined by 
\begin{equation*}
\delta(x,y) =\frac{|x_0y_1-y_0x_1|}{\max\{|x_0|,|x_1|\}\max\{|y_0|,|y_1|\}}
\end{equation*}
for $x=(x_0:x_1)$ and $ y=(y_0:y_1)$ in $\PP^1(L_p)$. (We suppress the dependence on $p$ in the above notation.) When $L_p$ is non-archimedean, $\delta$ is precisely the spherical metric on $\bP^1(L_p)$. In the case that $\nu$ admits a continuous potential with respect to $\lambda_p$, then by necessity the diagonal must be of $(\nu - \lambda_p)\otimes (\nu-\lambda_p)$-measure zero, and it is easy to see that the energy $I(\nu)$ corresponds exactly to the energy pairing of $\nu$ with the standard $p$-adic measure $\lambda_p$ in \eqref{eqn:FRL-pairing}:
\begin{equation}
 I(\nu) = ( \nu-\lambda_p, \nu - \lambda_p )_p \nonumber
\end{equation}

It follows from \cite[Theorem 1]{F-P-EIOLF} that for each $L_p$ there exists a unique minimal Borel probability measure $\mu_{L_p}$ such that
\begin{equation}\label{eqn:FP1} 
I(\nu)\geq I(\mu_{L_p})
\end{equation}
for every Borel probability measure $\nu$ supported on $\bP^1(L_p)$, with equality if and only if $\nu=\mu_{L_p}$. However, it is important to note that if $\nu$ has any point masses, then $I(\nu)=\infty$, as the diagonal cannot be excluded in the definition of $I$ in order for the main theorems from \cite{F-P-EIOLF} to apply. 

Favre and Rivera-Letelier demonstrate in \cite{FRL} that the Weil height of $\al$ can be written
\begin{equation}\label{eqn:FRL-weil-height}
 h(\al) = \frac{1}{2} \sum_{p\in M_\bQ} (\lambda_p - [\al], \lambda_p - [\al])_p 
\end{equation}
where $\lambda_p$ is a standard measures described above on the Berkovich analytic line $\sP^1(\bC_p)$ with the usual $p$-adic absolute value (or achimedean absolute value when $p=\infty$) and the measure $[\al]$ is defined by
\[
[\al] = \frac{1}{\abs{G_\bQ \al}} \sum_{z\in G_\bQ \al} \delta_z 
\]
where $G_\bQ$ denotes the usual absolute Galois group over $\bQ$ and $\delta_z$ the Dirac measure with point mass at $z$. We regard this as a measure on $\sP^1(\bC_p)$ by fixing for all primes $p$ an embedding from an algebraic closure of $\bQ$ to $\bC_p$.

The main idea behind the proof of Theorem \ref{thm:1} is to apply an inequality of the same type as \eqref{eqn:FP1} to the terms in \eqref{eqn:FRL-weil-height} for which $p\in S$ in order to get lower bounds on the local energy pairings. As the measures $[\al]$ consist purely of a finite sum of point masses, however, in order to apply our bound, we must first approximate the measure $[\al]$ by an appropriate regularization of the measure which admits a continuous potential. We give these regularizations in Section \ref{sec:reg-measures}. These regularizations are not typically supported in $L_p$, so in Section \ref{sec:pot-th-res} we will prove several results akin to \cite[Theorem 1]{F-P-EIOLF} providing lower bounds for the energy of an $\ep$-neighborhood around the line $\bP^1(L_p)$ for each $p\in S$. We will then use these lower bounds to prove the main results in Section \ref{sec:main-res}. 

In the final Section of this paper we will combine our results from Theorem \ref{thm:1} with the general lower height bound of Dobrowolski \cite{Dob} to achieve Theorem \ref{thm:2}.

\section{Regularized measures}\label{sec:reg-measures}
Suppose our $\al\in L_S$ as in the formulation of Theorem \ref{thm:1} for $p\in S$. We wish to find a regularization of the measure
\[
[\al] = \frac{1}{\abs{G_\bQ \al}} \sum_{z\in G_\bQ \al} \delta_z 
\]
supported on $\bP^1(L_p)$ which admits a continuous potential with respect to the standard measure $\lambda_p$. (This regularization may be supported on a larger space; for example, when $L_\infty = \bR$, our regularization will be supported on $\bC$.)

\subsection{Archimedean regularization of measures}
We start with the real case $L_\infty=\bR$. Our technique will be to replace the  point masses in the probability measure $[\al]$ defined above by measures which are suitably regular. Our approach is quite similar to that of \cite[\S 2]{FRL} but our choice of regularization is slightly simpler and achieves the minimal logarithmic energy possible. 

Specifically, for a given Dirac point mass $\delta_x$ for $x\in \bC$, we will define $\delta_{x,\ep}$ to be the normalized unit Lebesgue measure of the circle $\{ z\in\bC : \abs{z-x} = \ep\}$. For $F\subset\bC$ a finite set and $[F] = \nicefrac{1}{\abs{F}}\sum_{x\in F} \delta_x$, we will define
\[
 [F]_\ep = \frac{1}{\abs{F}} \sum_{x\in F} \delta_{x,\ep}.
\]
It is immediate that these measures admit a continuous potential as defined above in Section \ref{sec:background}.

We now prove a few easy lemmas regarding our regularized measures (cf. Lemmas 2.9 and 2.10 of \cite{FRL}). 
\begin{lemma}\label{lemma:arch-1}
Let $F\subset \bC$ be a finite set and $\ep>0$. Then 
\[
 \abs{([F],\lambda_\infty) - ([F]_\ep, \lambda_\infty)} \leq \ep,
\]
where $[F],[F]_\ep$ are the probability measures defined above.
\end{lemma}
\begin{proof}
 Our proof is essentially the same as that of \cite[Lemma 2.9]{FRL}. It suffices to prove the bound for a singleton $z\in F$ with measures $\delta_z, \delta_{z,\ep}$. We use the standard notation $\log^+$ for the function $\max\{\log,0\}$. Note that
 \[
  (\delta_{z,\ep},\lambda_\infty) - (\delta_{z},\lambda_\infty) = \int_0^1 \log^+\abs{z + \ep\cdot  e^{2\pi i t}}\,dt - \log^+\abs{z},
 \]
 but $ \abs{\log^+\abs{z + \ep\cdot  e^{2\pi i t}} - \log^+\abs{z}} \leq \ep$ for every real $t$. The result follows.
\end{proof}
\begin{lemma}\label{lemma:arch-2}
 Let $F\subset \bC$ be a finite set and $\ep>0$. Then 
\[
 ([F]_\ep , [F]_\ep) \leq ([F],[F]) - \frac{\log \ep}{\abs{F}}.
\]
\end{lemma}
Note that this lemma improves on \cite[Lemma 2.10]{FRL} as the term $C/\abs{F}$ on the right hand side is removed.
\begin{proof}
 We follow the same method as in the proof of \cite[Lemma 2.10]{FRL}. We note that for $\ep>0$ and two points $z\neq z'\in\bC$,
\begin{align*}
 -(\delta_{z,\ep},\delta_{z',\ep}) &= \int_0^1\int_0^1 \log\abs{z + \ep\cdot e^{2\pi it} - (z' + \ep\cdot e^{2\pi is})}\,dt\,ds\\
 & = \int_0^1\max\{ \log\abs{z- (z' + \ep\cdot e^{2\pi is})}, \log \ep\}\,ds\\
 &\geq \max\left\{  \int_0^1\log\abs{z- (z' + \ep\cdot e^{2\pi is})}\,ds, \log \ep\right\}\\
 &\geq \max\{\log\abs{z-z'},\log \ep\}\geq \log\abs{z-z'} = -(\delta_z,\delta_{z'})
\end{align*}
so for each $z\neq z'$, we have
$(\delta_z,\delta_{z'})\geq (\delta_{z,\ep},\delta_{z',\ep})$. On the other hand, we have what is essentially the logarithmic capacity:
\[
 (\delta_{z,\ep},\delta_{z,\ep}) = - \log \ep.
\]
Thus
\begin{align*}
 ([F]_\ep, [F]_\ep) &= \frac{1}{\abs{F}^2} \sum_{\substack{ z,z'\in F\\ z\neq z'}} (\delta_{z,\ep},\delta_{z',\ep}) + \frac{1}{\abs{F}^2} \sum_{z\in F} (\delta_{z,\ep}, \delta_{z,\ep}) \\
 &\leq \frac{1}{\abs{F}^2} \sum_{\substack{ z,z'\in F\\ z\neq z'}} (\delta_{z},\delta_{z'}) + \frac{1}{\abs{F}^2}\cdot \abs{F} \cdot(-\log \ep)\\
 &= ([F],[F]) - \frac{\log \ep}{\abs{F}}.\qedhere
\end{align*}
\end{proof}

We now prove our main result of this section:
\begin{prop}\label{prop:arch-approx}
Let $F\subset \bC$ be a finite set and $\ep>0$. Then 
\[
 ([F]-\lambda_\infty,\ [F]-\lambda_\infty) \geq ([F]_\ep-\lambda_\infty,\ [F]_\ep-\lambda_\infty) - 2\ep + \frac{\log \ep}{\abs{F}}.
\]
\end{prop}
\begin{proof}
 We use the bilinearity of the energy pairing to write
 \[
  ([F] - \lambda_\infty, [F]-\lambda_\infty) = ([F],[F]) - 2([F],\lambda_\infty) + (\lambda_\infty,\lambda_\infty)
 \]
(noting that these individual energy pairings must be finite by \cite[Lemma 4.3]{FRL}), and from this it follows that
 \[
 ([F] - \lambda_\infty, [F]-\lambda_\infty) = ([F]_\ep - \lambda_\infty, [F]_\ep-\lambda_\infty) + E_1 + E_2 
 \]
 where
 \[
 E_1 = -2([F],\lambda_\infty) +2([F]_\ep,\lambda_\infty), \quad\text{and}\quad E_2 = ([F],[F]) - ([F]_\ep,[F]_\ep).
 \]
 By applying Lemma \ref{lemma:arch-1} to $E_1$ and Lemma \ref{lemma:arch-2} to $E_2$, we get the desired bound.
\end{proof}

\subsection{Non-archimedean regularization of measures}\label{sec:non-arch-regularization}
We fix a rational prime $p$ and an associated local field $L_p$ of residue characteristic $p$. We will follow the regularization technique introduced in \cite[\S 4]{FRL}: for $x\in\bC_p$ and $\ep\in \bR$, $\ep > 0$, we let $\zeta_{x,\ep}$ denote the type II or type III point of the Berkovich affine line $\sA^1(\bC_p)$, corresponding to the disc of radius $\ep$ around $x$. For our measure $[\al]$ on $\sP^1(\bC_p)$ for $\al\in L_S$ we define the \emph{regularized measure $[\al]_\ep$} to be
\begin{equation}\label{eqn:reg-measure}
 [\al]_\ep = \frac{1}{\abs{G_\bQ \al}} \sum_{z\in G_\bQ \al} \delta_{z,\ep}.
\end{equation}
where $\delta_{z,\ep}$ denotes the Dirac unit point mass supported on the point $\zeta_{z,\ep}\in\sA^1(\bC_p)$.

We will always assume that $0<\ep<1$ in our regularization throughout this paper. Let $e=e(L_p/\bQ_p)$ and $f=f(L_p/\bQ_p)$ denote the local ramification and inertial degrees of $L_p$, respectively. As we have fixed the prime $p$ in this section, we will set $(\cdot,\cdot)=(\cdot,\cdot)_p$ for the energy pairing \eqref{eqn:FRL-pairing}. The following result provides an estimate on how far the energy pairing of $[\al]_\ep$ with the standard measure $\lambda_p$ is from the energy pairing of $[\al]$ with $\lambda_p$:

\begin{prop}[Favre, Rivera-Letelier \protect{\cite[Prop. 9]{FRL}}]\label{prop:non-arch-bound}
 For all $0<\ep<1$,
 \begin{equation}
  ([\al] - \lambda_p, [\al]-\lambda_p) \geq ([\al]_\ep - \lambda_p, [\al]_\ep-\lambda_p) + \frac{\log \ep}{\abs{G_\bQ \al}}.
 \end{equation}
\end{prop}

\section{Potential theoretic results}\label{sec:pot-th-res}
In both the non-archimedean and archimedean settings, our regularized measures lie in an $\ep$-neighborhood of the line $\bP^1(L_p)$ for $L_p/\bQ_p$ our chosen local field at each place, and thus the energies (the $\delta$-Robin constants) of the lines $\bP^1(L_p)$ computed in \cite{F-P-EIOLF} cannot be directly applied. In order to get a lower bound on the energy, we will compute in this section estimates on the energy of the $\ep$-neighborhoods of the appropriate lines.

\subsection{Archimedean results}
Our goal in this section is to prove a potential theoretic result which is perhaps of independent interest for the $-\log \delta(x,y)$ kernel which will be used to prove our main results in the archimedean case.
In \cite{F-P-EIOLF} it is shown that the $\delta$-Robin constant of $\bP^1(\bR)$ is given by $I(\mu_{\bR})$, where
\[
\mu_\bR(z) = \frac{1}{\pi^2 z}\log\left|\frac{z+1}{z-1}\right|\,dz.
\]
As our regularized measures for the archimedean places are not supported entirely on the projective real line, we cannot directly apply this result even when $F\subset \bR$ in order to obtain a lower bound on the local factors of the height. Instead, for each $0<\ep<1$, we will prove a bound on the $\delta$-Robin constant of the set
\[
E_\ep = \{ z \in\bC : \abs{\ON{Im}{z}}\leq \ep\}\cup\{\infty\}\subset \bP^1(\bC),
\]
on which our measures $[F]_\ep$ for a finite set $F\subset\bR$ are supported, and use this lower bound on the $\delta$-energy in an analogue of \cite[Theorem 1]{F-P-EIOLF}. Our result is the following:
\begin{thm}\label{thm:robin-constant}
 Let $E_\ep$ be as above for $0<\ep<1$. Then the $\delta$-Robin constant of $E_\ep$ satisfies 
\[ 
 V_\delta(E_\ep)\geq \frac{7\zeta(3)}{2\pi^2} - 0.95 \sqrt{\ep}.
\]
In particular, for any Borel probability measure $\mu$ supported on $E_\ep$, we have
\[
 I(\mu) = \iint_{E_\ep\times E_\ep} -\log\delta(x,y)\,d\mu(x)\,d\mu(y) \geq \frac{7\zeta(3)}{2\pi^2} - 0.95\sqrt{\ep}.
\] 
\end{thm}
Before proving this theorem, we will prove two technical lemmas which will be useful:
\begin{lemma}\label{lemma:techni1}
Let $y\leq 1$ be a positive real number. For all $z \in \RR\setminus\{0\}$ one has
\[
\frac{4y^2\log\left(1+\frac{y^2}{z^2}\right)}{z^2+y^2 z^2} \geq \left(\log\left(1+\frac{y^2}{z^2}\right)\right)^2.
\]
\end{lemma}
\begin{proof}[Proof of lemma]
Since every $z$ appears with an even power, we may assume that $z> 0$. After dividing both sides by the positive value $\log\left(1+\frac{y^2}{z^2}\right)$, it is enough to prove
\begin{equation}\label{eq:techni1}
\frac{4y^2}{z^2 +y^2 z^2} \geq \log\left(1+\frac{y^2}{z^2}\right).
\end{equation}
We regard both sides of \eqref{eq:techni1} as functions in $z$.  The derivative of $\frac{4y^2}{z^2 +y^2 z^2}$ (with respect to $z$) is $\frac{-8y^2}{(1+y^2)z^3}$ and the derivative of $\log\left(1+\frac{y^2}{z^2}\right)$ (with respect to $z$) is $\frac{-2y^2}{z(y^2+z^2)}$. Using $0 < y\leq 1$ we see that 
\[
\frac{-8y^2}{(1+y^2)z^3} \leq \frac{-8y^2}{2 z^3} \leq \frac{-2 y^2}{z z^2} \leq \frac{-2y^2}{z(y^2+z^2)} \quad \text{ for all } z>0.
\]
Since both functions in \eqref{eq:techni1} tend to zero as $z$ tends to infinity, this implies that the inequality \eqref{eq:techni1}, and hence the statement of the lemma, is true. 
\end{proof}
\begin{lemma}\label{lemma:technical-2}
Let $y\leq 1$ be a positive real number. For all $x\in \RR$ one has
\[
\int_{\RR} \log\left(1+\frac{y^2}{(x-z)^2}\right) \cdot \frac{1}{z \pi^2}\log\left\vert \frac{z+1}{z-1}\right\vert \, d z \leq \sqrt{y} \cdot 1.9.
\] 
\end{lemma}
\begin{proof}[Proof of lemma]
We start by applying the Cauchy-Schwarz-Bunyakovsky inequality to get
\begin{multline}\label{eq:hoelder}
\int_{\RR} \log\left(1+\frac{y^2}{(x-z)^2}\right) \cdot \frac{1}{z \pi^2}\log\left\vert \frac{z+1}{z-1}\right\vert \, d z \\ \leq \sqrt{\int_{\RR} \left(\log\left(1+\frac{y^2}{(x-z)^2}\right)\right)^2 \, d z} \cdot \sqrt{\int_{\RR} \left(\frac{1}{z \pi^2}\log\left\vert \frac{z+1}{z-1}\right\vert\right)^2 \, d z }
\end{multline}
The second factor is a constant $c_{\mathbb{R}}=0.450158\ldots$, so we are left to find an upper bound for the integral
\[
\int_{\RR} \left(\log\left(1+\frac{y^2}{(x-z)^2}\right)\right)^2 \, d z = \int_{\RR} \left(\log\left(1+\frac{y^2}{z^2}\right)\right)^2 \, d z .
\]
We want to regard this latter integral as a function in $y$. Therefore, we define
\[
F(y)=\int_{\RR} \left(\log\left(1+\frac{y^2}{z^2}\right)\right)^2 \, d z
\]
for all $y \in (0,1)$ (note that the integral defining $F(y)$ converges for every $y$). We claim that $\frac{F(y)}{y}$ is strictly increasing on $(0,1)$. By the Leibniz rule we get
\begin{align*}
\frac{d}{d y} \frac{F(y)}{y} & = \frac{y \int_{\RR} \frac{d}{d y} \left(\log\left(1+\frac{y^2}{z^2}\right)\right)^2 \, d z - \int_{\RR} \left(\log\left(1+\frac{y^2}{z^2}\right)\right)^2 \, d z}{y^2}\\
 & = \frac{\int_{\RR} \frac{4y^2\log\left(1+\frac{y^2}{z^2}\right)}{z^2 + y^2 z^2} - \left(\log\left(1+\frac{y^2}{z^2}\right)\right)^2 \, d z}{y^2}.
\end{align*}
By Lemma \ref{lemma:techni1} this is non-negative for all $0<y\leq 1$, proving the claim. It follows that $F(y) \leq y F(1) \leq y \cdot 17.420688\ldots$. From \eqref{eq:hoelder} we get
\[
\int_{\RR} \log\left(1+\frac{y^2}{(x-z)^2}\right) \cdot \frac{1}{\pi^2}\log\left\vert \frac{z+1}{z-1}\right\vert \, d z \leq \sqrt{y} \cdot \sqrt{F(1)} \cdot c_{\mathbb{R}} \leq 1.87887 \ldots \sqrt{y},
\]
which concludes the proof of this lemma.
\end{proof}
\noindent We are now ready to prove the theorem.
\begin{proof}[Proof of Theorem \ref{thm:robin-constant}]
Our proof relies on \cite[Theorem 7]{F-P-EIOLF}, namely, that for any closed set $E\subset\bP^1(\bC)$ of finite $\delta$-Robin constant $V_\delta(E)$ and any Borel probability measure $\nu$ supported on $E$, we have
\begin{equation}\label{eqn:thm-7-eqn}
 \inf_{z\in E} U^\nu_\delta(z) \leq V_\delta(E) \leq \sup_{z\in E} U^\nu_\delta(z)
\end{equation}
where $U^\nu_\delta$ is the $\delta$-potential associated to the measure $\nu$ given by
\[
 U^\nu_\delta(z) = \int_E -\log\delta(w,z)\,d\nu(w).
\]
The idea of the proof will be to study the decay of the potential associated to the minimal energy measure $\mu_\bR$ given in \cite[Theorem 1]{F-P-EIOLF}. Since $\mu_\bR$ is supported on $\bP^1(\bR)\subset E_\ep$, $E_\ep$ obviously has finite $\delta$-energy, and the infimum of $U^{\mu_\bR}_\delta$ on $E_\ep$ will determine a lower bound for the $\delta$-Robin constant of $E_\ep$.

Specifically, we note that, for $x+yi\in E_\ep\setminus\{\infty\}$, we have
\begin{align*}
 U^{\mu_\bR}_\delta(x+iy) &= \log^+\abs{x+iy} + \int_{\bR} \log^+\abs{z}\,d\mu_{\bR}(z) - \int_{\bR} \log\abs{x+iy-z}\,d\mu_\bR(z)\\
 &= \log^+\abs{x+iy} + \frac{7\zeta(3)}{2\pi^2} - \int_{\bR} \log\abs{x+iy-z}\,d\mu_\bR(z)
\end{align*}
where we have used the result $\int_{\bR} \log^+\abs{z}\,d\mu_{\bR}(z)= \nicefrac{7\zeta(3)}{2\pi^2}$ which follows from the computations in the proof of \cite[Theorem 1]{F-P-EIOLF}. Now, 
\begin{align*}
\int_{\bR} \log\abs{x+iy-z}\,d\mu_\bR(z) &= \frac{1}{2} \int_{\bR} \log\left((x-z)^2 + y^2\right)\,d\mu_\bR(z)\\ &=  \int_{\bR} \log\abs{x-z}\,d\mu_\bR(z) + \frac{1}{2} \int_{\bR} \log\left(1 + \frac{y^2}{(x-z)^2}\right)\,d\mu_\bR(z).
\end{align*}
One can check directly (via an analysis similar to that used in \cite{F-P-EIOLF}) that the quantity $\int_{\bR} \log\abs{x-z}\,d\mu_\bR(z)=0$ when $x=\pm 1$. Combining this with the proof in \cite[Theorem 1]{F-P-EIOLF} that for all $x\neq \pm 1$, $U^{\mu_\bR}_\delta(x) = \nicefrac{7\zeta(3)}{2\pi^2}$, we see that in fact $U^{\mu_\bR}_\delta(x) = \nicefrac{7\zeta(3)}{2\pi^2}$ for all $x\in\bR$, and so it follows that 
\[
 \log^+\abs{x} - \int_{\bR} \log\abs{x-z}\,d\mu_\bR(z) = 0\quad\text{for all}\quad x\in\bR.
\]
Using $\log^+\abs{x+iy}\geq \log^+\abs{x}$ we obtain:
\begin{equation}\label{eqn:potential-bd-1}
 U^{\mu_\bR}_\delta(x+iy) \geq \frac{7\zeta(3)}{2\pi^2} - \frac{1}{2} \int_{\bR} \log\left(1 + \frac{y^2}{(x-z)^2}\right)\,d\mu_\bR(z).
\end{equation}
By \eqref{eqn:thm-7-eqn} and Lemma \ref{lemma:technical-2} it follows immediately that
\[
V_\delta(E) \geq U^{\mu_\bR}_\delta(x+iy) \geq \frac{7\zeta(3)}{2\pi^2} - 0.95\cdot \sqrt{y} \geq \frac{7\zeta(3)}{2\pi^2} - 0.95\cdot \sqrt{\ep}.
\]
This completes the proof of Theorem \ref{thm:robin-constant}.
\end{proof}

\subsection{Non-archimedean results}
We will now prove analogous results for the $p$-adic setting. In particular, we assume a local field $L_p/\bQ_p$ is fixed, for a finite rational prime $p$ with notation as in \S \ref{sec:non-arch-regularization} above.

As before, we let $O=O_{L_p}$ denote the ring of integers of $L_p$ and we let $\pi$ be a uniformizing element. We also let $q=p^f$ denote the order of the residue field $O/\pi O$ and let $\sP^1(\bC_p)$ denote the Berkovich projective line. Notice that $\bP^1(L_p)$ is a compact subset of $\sP^1(\bC_p)$. We will in fact prove three results, determining the $\delta$-Robin constants for an $\ep$-neighborhood of $O$ in the sense of the retraction map of Favre and Rivera-Letelier \cite[\S 4]{FRL}, for a neighborhood of $O^\times$, and finally for a neighborhood of $\bP^1(L_p)$ itself.

In the following, we extend the kernel $\delta(x,y)$ from the classical projective line over our local field $\bP^1(L_p)$ to the Berkovich projective line $\sP^1(\bC_p)$ by viewing $-\log \delta(x,y)$ as the generalized Hsia kernel of \cite[\S 4.4]{BakerRumelyBook} with respect to the Gauss point $\zeta_{0,1}$ (the so-called \emph{spherical kernel}). As we did above, we use the notation $\zeta_{x,r}\in \sA^1(\bC_p)$ for the type II or type III point of the Berkovich projective line which corresponds to the sup norm on the disc of radius $r$ centered at $x\in\bC_p$. Note that $\sP^1(\bC_p) = \sA^1(\bC_p)\cup \{\infty\}$. 

\begin{thm}\label{thm:robin-constant-for-O-L}
 Let $n\in\bN\cup\{0\}$ and $0<\ep\leq 1$ satisfy $\abs{\pi^{n}}\leq \ep \leq \abs{\pi^{n-1}}$. Then the $\delta$-Robin constant of the set  
 \[
  F = \{ \zeta \in \sP^1(\bC_p) : \zeta = \zeta_{x,\eta}\text{ for }x\in O,\ 0\leq\eta\leq \ep\}.
 \]
 is given by 
 \begin{equation}\label{eqn:robin-constant-for-O-L}
  V_\delta(F) =  \frac{\log (\abs{\pi^n}/\ep)}{q^n} - \left(1 - \frac{1}{q^n}\right) \frac{\log \abs{\pi}}{q-1}
 \end{equation}
 and in particular satisfies the inequality:
 \begin{equation}\label{eqn:robin-ineq-for-O-L}
  -\bigg(1 - \frac{1}{q^{n-1}}\bigg)\frac{\log\,\abs{\pi}}{q-1} < V_\delta(F) 
  \leq -\bigg(1 - \frac{1}{q^{n}}\bigg)\frac{\log\,\abs{\pi}}{q-1}.
 \end{equation}
\end{thm}
\noindent The reader may wish to note for comparison the classical result (cf. \cite[Example 4.1.24]{RumelyBook}) that 
\[
 V_\delta(O) =  -\frac{\log\, \abs{\pi}}{q-1},
\]
 and that as $\ep\searrow 0$, our result limits from below to the classical value, as $\abs{\pi}<\abs{\pi^n}/\ep\leq 1$.
\begin{proof}
First we note that it suffices to prove the result for $\abs{\pi^{n}}\leq \ep < \abs{\pi^{n-1}}$: as $\ep \nearrow\abs{\pi^{n-1}}$, the formula for $V_\delta(F)$  limits continuously to the analogous result when $\ep = \abs{\pi^{n-1}}$ and $n$ is replaced by $n-1$.

 When $\ep=1$, recall that $F$ is precisely a tree properly contained in the unit Berkovich disc which is rooted at the Gauss point $\zeta_{0,1}$ and branches at every radius $r=\abs{\pi^{n-1}}$ for $n=1,2,\ldots$ towards the points $\zeta_{a,r}$ where $a$ runs over a set of coset representatives for $\cO/\pi^n \cO$, and whose branches terminate in the points $\zeta_{a,0}=a\in \cO$. When $\abs{\pi^{n}}\leq \ep < \abs{\pi^{n-1}}$, we are `cutting off' part of the root around $\zeta_{0,1}$, and $F$ now consists of $q^{n}$ different trees, each rooted in a point of the form $\zeta_{a,\ep}$ as $a$ runs over a set of representatives for $\cO/\pi^{n}\cO$. It follows that the (exterior) boundary of the set $F$ is precisely the roots of these disjoint trees, that is, the finite of points $\zeta_{a,\ep}$ where $a$ runs over a set of the $q^n$ coset representatives for $\cO/\pi^n \cO$:
 \[
  \p_e F = \{ \zeta_{a,\ep} : a\in O/\pi^n O\}
 \]
 (Notice that for $a,b\in O$, $\zeta_{a,\ep} = \zeta_{b,\ep}$ if $a\equiv b\mod \pi^n O$ since $\abs{\pi^n}\leq \ep$, so this notation makes sense.) By \cite[Prop. 6.8]{BakerRumelyBook}, the equilibrium measure of the set $F$ is supported on its exterior boundary. Thus the equilibrium measure must consist of point masses supported on this set of $q^n$ points. By the equivariance of the kernel $\delta(x,y)$ under the action of $GL_2(O)$ induced on the Berkovich line, and the uniqueness of the equilibrium measure of $F$, this measure must be equally supported on each of the $q^n$ points $\zeta_{a,\ep}$. Let $\mu$ denote this equilibrium measure. We can then take advantage of the fact that the $\delta$-potential function $U^\mu_\delta$ is constant everywhere on $F$ (again, by the equivariance of the kernel under the action induced by $f(z)=z+b\in\GL_2(O)$ for any $b\in O$) to compute $V_\delta(F)$ by evaluating the potential at any given point in $F$, say, $\zeta_{0,\ep}\in F$:
 \[
  V_\delta(F) = U^\mu_\delta(\zeta_{0,\ep}) = \int_F -\log\delta(\zeta_{0,\ep},\xi)\,d\mu(\xi).
 \]
 Now for $a\in O$, we can compute the distances:
 \begin{equation*}\label{eqn:computing-spherical-kernel}
  \delta(\zeta_{0,\ep},\zeta_{a,\ep}) = \begin{dcases}
                    \ep      & \text{if }a\equiv 0 \mod \pi^n O\\
                    \abs{0-a}=\abs{a} & \text{if }a\not\equiv 0 \mod \pi^n O.
                             \end{dcases}
 \end{equation*}
 For our set of coset representatives, there are $q^{n-k} - q^{n-k-1}$ terms of absolute value $\abs{\pi^k}$, where $0\leq k \leq n$, and thus:
 \begin{equation}\label{eqn:V-delta-F-computation}
 \begin{split}
  V_\delta(F) = U^{\mu}_\delta(\zeta_{0,\ep}) &= \frac{1}{q^n}\sum_{a\in \cO/\pi^n O} -\log\delta(\zeta_{0,\ep},\zeta_{a,\ep})\\
  &= \frac{-1}{q^n}\left(\log \ep + \sum_{k=0}^{n-1} (q^{n-k}-q^{n-k-1})\log \abs{\pi^k}\right)\\
  &= \frac{\log (\abs{\pi^n}/\ep)}{q^n} - \left(1 - \frac{1}{q^n}\right) \frac{\log \abs{\pi}}{q-1}
 \end{split}
 \end{equation}
 which is our desired result.
 
 The inequality in the theorem follows by noting that $\abs{\pi^n}\leq \ep < \abs{\pi^{n-1}}$ implies that $\abs{\pi}< \abs{\pi^n}/\ep \leq 1$.
\end{proof}

\begin{thm}\label{thm:robin-constant-for-O-L-cross}
 Let $n\in\bN\cup\{0\}$ and $0<\ep\leq 1$ satisfy $\abs{\pi^{n}}\leq \ep \leq \abs{\pi^{n-1}}$. Then the $\delta$-Robin constant of the set  
 \[
  F' = \{ \zeta \in \sP^1(\bC_p) : \zeta = \zeta_{x,\eta}\text{ for }x\in O^\times,\ 0\leq\eta\leq \ep\}.
 \]
 is given by
 \begin{equation}\label{eqn:robin-constant-for-O-L-cross}
  V_\delta(F') = \frac{\log(\abs{\pi^n}/\ep)}{q^{n-1}(q-1)} - \left(1-\frac{1}{q^n}\right) \frac{q \log\,\abs{\pi}}{(q-1)^2} 
 \end{equation}
 and satisfies the inequality
 \begin{equation}\label{eqn:robin-ineq-for-O-L-cross}
 - \left(1-\frac{1}{q^{n-1}}\right) \frac{q \log\,\abs{\pi}}{(q-1)^2} < V_\delta(F') \leq - \left(1-\frac{1}{q^n}\right) \frac{q \log\,\abs{\pi}}{(q-1)^2}  .
 \end{equation}
\end{thm}
\noindent The reader may observe that the main term above matches the classical Robin constant of $O^\times$, which is:
\[
 V_\delta(O^\times) = - \frac{q\log\,\abs{\pi}}{(q-1)^2}.
\]
\begin{proof}
 Our proof follows the same reasoning as in the proof of the previous theorem. Again we may assume $\abs{\pi^{n}}\leq \ep < \abs{\pi^{n-1}}$. We start by noting that the exterior boundary of the set is given by
 \[
  \p_e F' = \{ \zeta_{a,\ep} : a\in O/\pi^n O\text{, and }\abs{a}=1 \}
 \]
 This consists of $q^n \cdot (q-1)/q = q^n - q^{n-1}$ points in $\sP^1(\bC_p)$. Let $f_b(z) = bz\in GL_2(O)$ for $b\in O^\times$. Then our kernel is equivariant under the induced action on the Berkovich line by $f_b$, and thus we must have that the equilibrium measure is equally supported on these $q^n-q^{n-1}$ points. Let $\nu$ denote this equilibrium measure. As in the previous proof, we will compute the $\delta$-Robin constant by evaluating the potential function $U_\delta^\nu$ at a specific point in $F'$, given that it is constant on $F'$. We choose to compute $V_\delta(F')=U_\delta^\nu(\zeta_{1,\ep})$. We note that for each $a\in O^\times$, there exists an integer $0\leq k \leq n$ such that $a \in (1 + \pi^k O)\setminus ( 1 + \pi^{k+1} O)$, and that $k=n$ precisely when $a\equiv 1 \mod \pi^n O$. Therefore, we can compute:
 \begin{equation*}\label{eqn:distances-in-O-L-cross}
  \delta(\zeta_{1,\ep},\zeta_{a,\ep}) = \begin{dcases}
                    \ep      & \text{if }a\equiv 1 \mod \pi^n O\\
                    \abs{1-a} = \abs{\pi^k} & \text{if }a \in (1 + \pi^k O)\setminus ( 1 + \pi^{k+1} O).
                             \end{dcases}
 \end{equation*}
 It remains to count how many elements of $\p_e F'$ lie in each set, but this is an easy exercise: there is one element supported in $1 + \pi^n O$, $q^{n-k}-q^{n-k-1}$ elements in $(1+\pi^k O)\setminus (1+\pi^{k+1} O)$ for each $1\leq k\leq n-1$, and $q^n - 2q^{n-1}$ elements in $O^\times \setminus (1+\pi O) \subset (1+O)\setminus (\pi O \cup (1+\pi O))$, the last of which we can ignore as we have $-\log \delta(\zeta_{1,\ep},\zeta)=0$ for such terms. Therefore,
 \begin{align*}
 V_\delta(F') &= U^\nu_\delta(\zeta_{1,\ep}) = \frac{1}{q^{n-1}(q-1)} \sum_{\xi \in \p_e F'} -\log\,\delta(\zeta_{1,\ep}, \xi)\\
 &= \frac{1}{q^{n-1}(q-1)} \bigg( -\log \ep - \sum_{k=1}^{n-1} \log\,\abs{\pi^k}(q^{n-k}-q^{n-k-1})\bigg)\\
 &= \frac{-\log \ep}{q^{n-1}(q-1)} - \log\,\abs{\pi}\sum_{k=1}^{n-1} kq^{-k}\\
 &= \frac{\log(\abs{\pi^n}/\ep)}{q^{n-1}(q-1)} - \left(1-\frac{1}{q^n}\right) \frac{q \log\,\abs{\pi}}{(q-1)^2}
 \end{align*}
 which gives the desired value of the $\delta$-Robin constant. The inequality now follows from observing as before that $\abs{\pi} < \abs{\pi^n}/\ep \leq 1$.
\end{proof}

\begin{thm}\label{thm:robin-constant-for-L-p}
 Fix $n\in\bN\cup \{0\}$ and let $0<\ep\leq 1$ satisfy $\abs{\pi^{n}}\leq \ep \leq\abs{\pi^{n-1}}$. Then the $\delta$-Robin constant of the set 
 \[
  G = \{ \zeta \in \sP^1(\bC_p) : \zeta = \zeta_{x,\eta}\text{ for }x\in L_p,\ 0\leq\eta\leq \ep\} \cup \{\infty\}
 \]
satisfies
\begin{equation}\label{eqn:robin-constant-for-L-p-ineq}
 \begin{split}
  V_\delta(G) & \geq \frac{\log (\abs{\pi^n}/\ep)}{q^{n-1}(q+1)} -\bigg(1 - \frac{1}{q^{n}}\bigg) \frac{q \log\,\abs{\pi}}{q^2-1}\\
  &> -\bigg(1 - \frac{1}{q^{n-1}}\bigg) \frac{q \log\,\abs{\pi}}{q^2-1}.
 \end{split}
 \end{equation}
\end{thm}
\noindent The reader may wish to note that as $\ep \searrow 0$, our result limits to the value obtained in \cite[Thm. 1(c)]{F-P-EIOLF} for the projective line over the local field $L_p$:
\[
 V_\delta(\bP^1(L_p)) = -\frac{q \log\,\abs{\pi}}{q^2-1}.
\]
\begin{proof}
Again we may assume $\abs{\pi^{n}}\leq \ep < \abs{\pi^{n-1}}$. Let $\iota(z) = 1/z \in GL_2(O)$, and denote by $\iota_*$ the induced action on $\sP^1(\bC_p)$. Notice that, for type II and III points $\zeta_{a,r}\in \sP^1(\bC_p)$ with $r<\abs{a}$, we have $\iota_*(\zeta_{a,r})= \zeta_{1/a,r/\abs{a}^2}$, since $\iota$ sends the closed disc $D(a,r)\subset\bC_p$ centered at $a$ of radius $r$ to the disc $D(1/a,r/\abs{a}^2)$ and vice versa. On the other hand, when $\abs{a}<r$, $\zeta_{a,r}=\zeta_{0,r}$ and $\iota_*(\zeta_{0,r}) = \zeta_{0,1/r}$. (The interested reader may wish to consult \cite[Lemma 2.4]{BakerRumelyBook} to compare the action on Berkovich discs.)
 
 Let 
 \[
  F = \{ \zeta \in \sP^1(\bC_p) : \zeta = \zeta_{x,\eta}\text{ for }x\in O,\ 0\leq\eta\leq \ep\}
 \]
as above in Theorem \ref{thm:robin-constant-for-O-L}, and 
\[
 F_\pi = \{ \zeta \in \sP^1(\bC_p) : \zeta = \zeta_{x,\eta}\text{ for }x\in \pi O,\ 0\leq\eta\leq \ep\}.
\]
 Let $G_\infty = G \setminus F$. Then each point $\zeta_{a,r}\in G_\infty$ has $\abs{a}= \abs{\pi^{-n}}>1$ for some integer $n\geq 1$. Notice that each point $\zeta_{a,r}\in G_\infty\setminus \{\infty\}$ has $r\leq \ep \leq 1 < \abs{a}$, so $\iota_*(\zeta_{a,r})= \zeta_{1/a,r/\abs{a}^2}\in F_\pi$ as $r/\abs{a}^2< r\leq \ep$, and $1/a\in \pi O$. In particular, as $\iota_*(\infty) = \zeta_{0,0}\in F_\pi$ as well, we have 
 \[
  \iota_*(G_\infty) \subset F_\pi.
 \]
 Since $\iota_*\circ \iota_*$ is the identity map, it follows that $G_\infty\subset \iota_*(F_\pi)$. 
 
 Let 
 $
  G^+ = F \cup \iota_*(F_\pi).
 $
 By the above, we have $G\subset G^+$, so it follows that $V_\delta(G)\geq V_\delta(G^+)$. We will obtain our result by computing the $\delta$-Robin constant of $G^+$. First, we make the observation that:
 \[
  V_\delta(F_\pi) = q V_\delta(F),
 \]
 which follows via the same analysis as in the proof of Theorem \ref{thm:robin-constant-for-O-L} that the equilibrium distribution is supported on the exterior points, which are precisely $\p_e F_\pi = \{ \zeta_{a,r} : a\in \pi O/\pi^n O \}$, that the equilibrium distribution must again be equally supported on each of these points, and observing in equation \eqref{eqn:V-delta-F-computation} that $\delta(\zeta_{0,\ep}, \zeta_{a,\ep})= 1$ for all $a\not\in \pi O$ so the computation is only rescaled by the weight for the new measure.
 
 Now, $G^+$ is stable under $\iota\in GL_2(O)$, and further, it is a union of $q+1$ disjoint copies of $F_\pi$ under translations by $GL_2(O)$. As we observed in Theorem \ref{thm:robin-constant-for-O-L}, the generalized Hsia kernel $-\log\,\delta(\cdot,\cdot)$ is equivariant under $GL_2(O)$, and for $\zeta, \xi$ in disjoint $GL_2(O)$-translations of $F_\pi$, we have $\delta(\zeta,\xi) = 1$. It follows that the equilibrium measure of $G^+$ must be equally supported on the exterior boundary of $G^+$, which is simply the union of the exterior boundaries of each disjoint $GL_2(O)$-translate copy of $F_\pi$. Notice that $\iota_*$ maps roots of the trees to roots, so the boundary points of $\iota_*(F_\pi)$ are simply the image under $\iota_*$ of the boundary points of $F_\pi$, $\p_e F_\pi = \{ \zeta_{a,\ep} : a\in \pi O/ \pi^n O\}$. Thus the equilibrium distribution of $G^+$ is simply the discrete measure equally supported on each of the $(q+1) q^{n-1}$ boundary points of $\p_e G^+$. Therefore 
 \begin{equation}\label{G-plus-Robin-constant}
  V_\delta(G^+) = \frac{1}{q+1} V_\delta(F_\pi) = \frac{q}{q+1} V_\delta(F) = \frac{\log (\abs{\pi^n}/\ep)}{q^{n-1}(q+1)} -\bigg(1 - \frac{1}{q^{n}}\bigg) \frac{q \log\,\abs{\pi}}{q^2-1}.
 \end{equation}
 The first inequality in the theorem statement now follows from the fact that $V_\delta(G)\geq V_\delta(G^+)$, and the second inequality from noting that $\abs{\pi} < \abs{\pi^n}/\ep$.
\end{proof}

\section{Proof of Theorem 1 and generalization}\label{sec:main-res}
For simplicity of notation we will prove first Theorem \ref{thm:1} as stated, which takes as the base field $\bQ$. We will then give the statement of the generalization of this result over a base field $K$ in Theorem \ref{thm:general-thm-1} below, and describe what routine changes need to be made for the proof to work over an arbitrary base field.

\begin{proof}[Proof of Theorem \ref{thm:1}] To ease notation, we will only prove \eqref{eqn:bound} and \eqref{eqn:bound-integer}. To prove \eqref{eqn:bound-unit} one only has to replace Theorem \ref{thm:robin-constant-for-O-L} by Theorem \ref{thm:robin-constant-for-O-L-cross} in the following argumentation.
Recall that 
\[
 h(\al) = \frac{1}{2} \sum_{p\in M_\bQ} (\lambda_p - [\al], \lambda_p - [\al])_p .
\]
For all finite $p$, by \cite[Lemma 5.4]{FRL} we have 
\[
(\lambda_p - [\al], \lambda_p - [\al])_p \geq 0,
\]
therefore, we can say 
\begin{equation}\label{eqn:h-inequality}
h(\al) \geq \frac{1}{2} \sum_{\substack{p\in S\\ p\nmid \infty,\ p^{1/e} < d}} (\lambda_p - [\al], \lambda_p - [\al])_p + \frac 12 (\lambda_\infty - [\al], \lambda_\infty - [\al])_\infty,
\end{equation}
where $d = \abs{G_{\QQ}\alpha}$ denotes the degree of $\alpha$ over $\QQ$. Let $\ep = \nicefrac{1}{d}$. For each finite place $p\in S$ with $p^{1/e} < d$, we let
\[
 n_p = \left\lfloor\frac{\log d}{\log p^{\nicefrac{1}{e}}}\right\rfloor = \left\lfloor\frac{e \log d}{\log p}\right\rfloor \in \mathbb{N}\cup \{0\}.
\]
Then, for $\pi$ a uniformizing parameter of $L_p$, the constant $\ep$ satisfies 
\[
 \abs{\pi^{n_p+1}}< \ep \leq \abs{\pi^{n_p}}\leq 1.
\]
The condition that $p^{1/e} < d$ ensures that $n_p\geq 1$, which is necessary to apply Theorems \ref{thm:robin-constant-for-O-L} and \ref{thm:robin-constant-for-L-p}.

Let $[\al]_\ep$ be the regularized measure on $\sP^1(\bC_p)$ as defined in Section \ref{sec:non-arch-regularization} above. Suppose that the conjugates of $\al$ all lie in $O_{L_p}$ (respectively, $L_p$). Then the measure $[\al]_\ep$ is supported on the set $F$ (resp. $G$) defined in Theorem \ref{thm:robin-constant-for-O-L} (resp. Theorem \ref{thm:robin-constant-for-L-p}) above, and further,
\[
 (\lambda_p - [\al]_\ep, \lambda_p - [\al]_\ep)_p = I([\al]_\ep)\geq 
 \begin{dcases}
 \displaystyle V_\delta(F) \geq (1-\frac{1}{q^{n_p}}) \frac{\log p}{e(q-1)} & \\
  \displaystyle V_\delta(G) \geq (1-\frac{1}{q^{n_p}})\dfrac{q \log p}{e(q^2-1)}, &\text{resp.,}
 \end{dcases}
\]
where we have used the fact that $\log\,\abs{\pi}_p = \nicefrac{-(\log p)}{e}$. Applying Proposition \ref{prop:non-arch-bound} to the measures $[\al]$ and $[\al]_\ep$ we obtain:
\begin{equation}\label{eqn:finite-p-bound, integer}
\sum_{\substack{p\in S\\ p\nmid \infty,\ p^{1/e} < d}} (\lambda_p - [\al], \lambda_p - [\al])_p \geq \sum_{\substack{p\in S\\ p\nmid\infty,\ p^{1/e} < d }} \left((1 - \frac{1}{q^{n_p}}) \frac{\log p}{e(q-1)}  - \frac{\log d}{d} \right)
\end{equation}
if all conjugates of $\al$ lie in $O_{L_p}$, and
\begin{equation}\label{eqn:finite-p-bound}
\sum_{\substack{p\in S\\ p\nmid \infty,\ p^{1/e} < d}} (\lambda_p - [\al], \lambda_p - [\al])_p \geq \sum_{\substack{p\in S\\ p\nmid\infty,\ p^{1/e} < d}} \left((1 - \frac{1}{q^{n_p}}) \dfrac{q \log p}{e(q^2-1)} - \frac{\log d}{d} \right)
\end{equation}
else.

The remainder of the proof now has two cases, depending on whether $\infty\in S$ (in which case our number is totally real, as all conjugates lie in $L_\infty=\bR$) or not. First, suppose $\infty\notin S$. We use M. Baker's reformulation \cite{BakerAverages} of Mahler's inequality \cite{Mahler} to bound the archimedean term of the height:
\begin{equation}\label{eqn:mahler}
(\lambda_\infty - [\al], \lambda_\infty - [\al])_\infty\geq -\frac{\log d}{d-1},
\end{equation}
and at the places of $S$, we use our local bounds at each $p$. If on the other hand we have $\infty\in S$, then we apply Proposition \ref{prop:arch-approx} to say that 
\[
 (\lambda_\infty - [\al], \lambda_\infty - [\al])_\infty \geq (\lambda_\infty - [\al]_\ep, \lambda_\infty - [\al]_\ep)_\infty - 2\ep + \frac{\log \ep}{d}. 
\]
Now, $[\al]_\ep$ has support in the strip $E_\ep\subset\bP^1(\bC)$ with imaginary part bounded in absolute value by $\ep$, so applying Theorem \ref{thm:robin-constant} to $[\al]_\ep$, we see that:
\[
(\lambda_\infty - [\al]_\ep, \lambda_\infty - [\al]_\ep)_\infty = I([\al]_\ep) \geq V_\delta(E_\ep) \geq \frac{7\zeta(3)}{2\pi^2} - 0.95 \sqrt{\ep}.
\]
Take $\ep = \frac{1}{d^2}$ and set
\[
V_{\infty}=\frac{7\zeta(3)}{4\pi^2} - \frac{0.95 d + 2}{2 d^2} - \frac{(d-2) \log{d}}{2 d (d-1)}.
\]
Then combining these two results, we see that
\begin{align}\label{eqn:arch-bound}
(\lambda_\infty - [\al], \lambda_\infty - [\al])_\infty \geq \frac{7\zeta(3)}{2\pi^2} - \frac{0.95 d +2}{d^2} - \frac{2\log{d}}{d} = -\frac{\log{d}}{d-1} + 2V_{\infty}.
\end{align}
Of course we can ignore the fact that $\infty \in S$, when we do not benefit from this contribution to the height. Hence, we can replace $2 V_{\infty}$ in \eqref{eqn:arch-bound} by $\max\{2V_{\infty},0\}$.
Combining the above inequalities \eqref{eqn:finite-p-bound, integer}, respectively \eqref{eqn:finite-p-bound}, and \eqref{eqn:mahler} or \eqref{eqn:arch-bound} with equation \eqref{eqn:h-inequality} gives the desired result.
\end{proof}

We will now state the generalization of Theorem \ref{thm:1} to arbitrary base number field $K$. 
\begin{thm}\label{thm:general-thm-1}
 Fix a number field $K$ and $M_K$ denote the set of places of $K$. Let $S\subset M_K$ and for each $v\in S$ suppose we are given a local finite normal extension $L_v/K_v$, taking $L_v=\bR$ if $v\mid\infty$. For $v\nmid\infty$ we let
\begin{center}
\begin{tabular}{cl}
 $q_v$ & the order of the residue field of $L_v$,\\
 $p_v$ & the rational prime over which $v$ lies, and\\
 $e_v$ & the local ramification index of $L_v/\bQ_v$.
\end{tabular}
\end{center}
Let $L_S$ denote the field of all $\al\in\overline{K}$ for which the minimal polynomial of $\al$ splits over $L_v$ for every $v\in S$ and let $N_v = \nicefrac{[K_v:\bQ_v]}{[K:\bQ]}$ for each $v\in S$.
 Suppose that $\al\in L_S$ with $d=[K(\al):K]>1$, and let $V_{\infty}$ be the constant, depending on $d$, from Theorem \ref{thm:1}. Then
\begin{align*}
 h(\al) \geq - \sum_{\substack{v\notin S\\ v\mid\infty}} \frac{N_v \log d}{2(d-1)} + \sum_{\substack{v\in S\\ v\mid\infty}} N_v V_{\infty}   + \frac{1}{2} \sum_{\substack{v\in S\\ v\nmid\infty,\ p_v^{1/e_v} < d}} N_v\cdot \left( (1- \frac{1}{q_v^{n_v}})\frac{q_v \log{p}}{e_v ( q_v^2 -1 )} - \frac{\log d}{d}\right)
\end{align*}
where
\[
 n_v = \left\lfloor \frac{e_v \log d}{\log p_v}\right\rfloor.
\]
If in addition $\al$ is an algebraic integer, then we have
\begin{align*}
 h(\al) \geq - \sum_{\substack{v\notin S\\ v\mid\infty}} \frac{N_v \log d}{2(d-1)} + \sum_{\substack{v\in S\\ v\mid\infty}} N_v V_{\infty} + \frac{1}{2} \sum_{\substack{v\in S\\ v \nmid \infty,\ p_v^{1/e_v}< d}} N_v \left((1 - \frac{1}{q_v^{n_v}}) \frac{\log{p_v}}{e_v(q_v -1)} - \frac{\log d}{d} \right).
\end{align*}

\end{thm}
\begin{proof}[Sketch of proof]
The proof is essentially the same as that for Theorem \ref{thm:1} above with the usual  modifications: we start with the standard normalized expression for the Weil height written over $K$:
\[
 h(\al) = \frac12 \sum_{v\in M_K} N_v \cdot ([\al]-\lambda_v,[\al]-\lambda_v)_v
\]
where above the $v$-adic absolute value in the energy pairing is chosen to extend the usual absolute value on $\bQ$ over which it lies, and apply the same bounds locally as above, but when applying Theorems \ref{thm:robin-constant-for-O-L} and \ref{thm:robin-constant-for-L-p} for finite places in computing the $\delta$-Robin constants we replace the former residue field order $q$ of $L_p$ with the order $q_v$ of the residue field $L_v$ and we observe that $\log\,\abs{\pi}_v = \nicefrac{-\log p_v}{e_v}$ where $\pi$ is the uniformizing parameter of $L_v$. The remainder of the proof is unchanged.
\end{proof}

\section{Proof of Theorem \ref{thm:2}}

Now we will use Theorem \ref{thm:1} to calculate absolute lower bounds for non roots of unity in $L_S^{\times}$. We can use Schinzel's lower bound mentioned in the introduction, whenever $\infty \in S$. Hence, for simplicity  we assume from now on that $\infty \notin S$. For totally $p$-adic algebraic units, Petsche \cite{PetscheSmallUnits} gave an easy argument to verify an effective lower height bound. We will summarize and extend his idea in our setting.

\begin{prop}[Petsche] \label{Petsche}
Let $S$ be a finite set of primes, and $\al \in L_{S}^{\times}$ be an algebraic unit which is not a root of unity. Moreover, let $l$ be a common multiple of $2$ and the elements in $\{p^{f_p} -1 \vert p\in S\}$. Then we have
\[
h(\al)\geq \frac{1}{l} \left( \sum_{p\in S} \frac{v_p (l)+1}{e_p}\log p - \log 2 \right).
\]
\end{prop}

\begin{proof}
Let $K$ be a normal extension of $\bQ$, with $\alpha \in K$ and $e_p = e(K,p)$, $f_p=f(K,p)$ for all $p\in S$.
Fix a prime $p \in S$ and let $\mathfrak{p}$ be any prime in $K$ lying above $p$. By $N_{K/\bQ}(.)$ we denote the norm on the field extension $K/\bQ$ and set $q=p^{f_p}=N_{K/\bQ}(\mathfrak{p})$.
Then, since no power of $\alpha$ is in $\mathfrak{p}$ and $\al$ is not a root of unity, we have $\alpha^{q-1}-1 \in \mathfrak{p}\setminus\{0\}$.
Therefore, for any $n\in\bN$ we have
\begin{equation}\label{primeequiv}
\al^{(q-1)p^n}-1=((\alpha^{q-1}-1)+1)^{p^n} -1 = \sum_{i=1}^{p^n}\binom{p^n}{i}(\alpha^{q-1}-1)^i \equiv 0 \mod \mathfrak{p}^{n+1}.
\end{equation}
Here we have used the well known formula $v_p(\binom{p^n}{i})=n-v_p(i)\geq n-i+1$.

Let $l\in \bN$ be as in the assumption. Since \eqref{primeequiv} is true for all pairs $\mathfrak{p}\vert p$, with $p \in S$, and $K/\bQ$ is Galois we get
\[
\vert N_{K/\bQ}(\al^{l}-1)\vert \geq \prod_{p\in S} \prod_{\mathfrak{p}\mid p} N_{K/\bQ}(\mathfrak{p})^{v_p (l)+1} = \left(\prod_{p\in S} p^{\nicefrac{v_p (l)+1}{e_p}}\right)^{[K:\bQ]} .
\]
By basic height estimates it follows
\[
l h(\alpha) + \log 2 = h(\al^l) + \log 2 \geq h(\alpha^{l}-1)\geq \sum_{p\in S} \frac{v_p (l)+1}{e_p}\log p .
\]
The statement of the proposition follows immediately.
\end{proof}

Dubickas and Mossinghoff \cite{DubickasMossinghoff} have slightly strengthened Petsche's result. However, both estimates are only positive for small ramification degrees, and in the special case of an algebraic unit. In the remainder of this paper we calculate lower height bounds for all elements in $L_S^\times$ which are not roots of unity. 

Therefore we define real valued functions on the interval $(1,\infty)$ by
\[
f_S (x)= - \frac{\log x}{2(x-1)} + \frac{1}{2} \sum_{\substack{p\in S\\  p^{1/e} < x}} \left((1 - \frac{1}{q^{n_p}}) \frac{q\log{p}}{e(q^2 -1)} - \frac{\log x}{x} \right)
\] and
\[
g_S (x)= - \frac{\log x}{2(x-1)}  + \frac{1}{2} \sum_{\substack{p\in S \\ p^{1/e} < x}} \left((1 - \frac{1}{q^{n_p}}) \frac{q\log{p}}{e(q -1)^2} - \frac{\log x}{x} \right),
\]
where $n_p = \left\lfloor \nicefrac{e \log x}{\log p}\right\rfloor$.    

\begin{lemma}\label{lem:1}
 Let $\al \in L_S^\times$ be not an algebraic unit and let $y\in (1,\infty)$ be such that $f_S(y) \geq \nicefrac{\log(2)}{y}$. Then we have
\begin{equation}\label{minmaxeq-nounit}
 h(\alpha)\geq \min\left\{ f_{S'}(\lceil y \rceil),\frac{\log{2}}{\lfloor y \rfloor} \right\}\geq \frac{\log{2}}{ y },
\end{equation}
where $S'=\{p\in S\vert p^{1/e}\leq y\}$.
\end{lemma}

\begin{proof}
The function $f_S (x)$ tends to a positive value as $x$ tends to infinity and is negative in the interval $(1,\ep)$ for $\ep= \min_{p\in S}\{p^{1/e}\}$. Since $\nicefrac{\log{2}}{x}$ is a positive monotonically decreasing function approaching zero, a $y$ as in the assumption exists. Set $S' = \{p\in S\vert p^{1/e}\leq y\}$, then $S'$ is not empty and we have $\alpha \in L_S \subseteq L_{S'}$. Moreover, $f_{S'}(x)$ is monotonically increasing in the interval $[y,\infty)$. In particular, we have $f_{S'}(x) \geq \nicefrac{\log{2}}{x}$ for all $x \geq y$.

Denote the degree of $\al$ by $d$. It is well known that the height of $\al$ is bounded from below by $\nicefrac{\log{2}}{d}$ and equation \eqref{eqn:bound} tells us $h(\al) \geq f_{S'}(d)$. Therefore we have
\[
 h(\al) \geq \min_{n\in\bN} \left\{\max\left\{f_{S'} (n),\frac{\log{2}}{n}\right\}\right\}. 
\]
For $n\geq y$ we have $f_{S'}(n) \geq f_{S'} (\lceil y \rceil)$ and for $n \leq y$ we have $\nicefrac{\log{2}}{n} \geq \nicefrac{\log{2}}{\lfloor y \rfloor}$. This proves the lemma, as the last inequality in \eqref{minmaxeq-nounit} is trivial.
\end{proof}

Let $\theta= 1.324...$ denote the smallest Pisot number; i.e. the real root of $x^3 -x -1$. We define the function $\Do(x):(1,\infty)\rightarrow \bR$ as
\[
 \Do(x)=\begin{dcases}
       \frac{\log{\theta}}{x} &\text{ for } x \leq 7 \\
        \frac{1}{4x}\left(\frac{\log\log x}{\log x}\right)^3 &\text{ for } x > 7
      \end{dcases}.
\]

\begin{lemma}\label{lem:2}
 Let $\al \in L_S^\times$ be an algebraic unit which is not a root of unity  and let $z \in (1,\infty)$ be such that $g_S (z) \geq \Do(z)$. Then we have
\begin{equation}\label{minmaxeq-unit}
 h(\alpha)\geq \min\left\{ g_{S'}(\lceil z \rceil), \Do(\lfloor z \rfloor) \right\} \geq \Do(z),
\end{equation}
where $S'=\{p\in S \vert p^{1/e} \leq z\}$.
\end{lemma}

\begin{proof}
We denote the degree of $\al$ by $d$.
The best known general lower bound for the height of $\al$ is due to Voutier \cite{VoutierDobrowolski}, who improved the constant of a previous lower bound due to Dobrowolski \cite{Dob}. This bound is
\begin{equation}\label{Dobrowolski}
 h(\al)\geq \frac{1}{4d}\left(\frac{\log\log d}{\log d}\right)^3.
\end{equation}
Using a complete list of algebraic units of degree $\leq 37$ of small Mahler measure due to Flammang, Rhin, and Sac-\'Ep\'ee \cite{FlammangRhinSac}, we find $h(\al) \geq \nicefrac{\log{\theta}}{d}$, whenever $d \leq 7$. This leads to the estimate $h(\al) \geq \Do(d)$. By \eqref{eqn:bound-unit} we also know $h(\al)\geq g_{S'} (d)$. Since $\Do(x)$ is monotonically decreasing, the Lemma follows with the same argument as in Lemma \ref{lem:1}.
\end{proof}

In order to prove Theorem \ref{thm:2} it remains to find elements $y$ and $z$ satisfying the assumptions of Lemmas \ref{lem:1} and \ref{lem:2}.

\begin{example}
Let $S=\{2,3\}$ and set $L_{p} = \bQ_p$ for $p\in S$. Then $L_S$ is the subfield of $\Qbar$ consisting of all algebraic numbers which are totally $2$-adic and totally $3$-adic. Note in the following, that by classical algebraic number theory there are no non-trivial totally $2$-adic roots of unity. With Proposition \ref{Petsche} we find that the height of an algebraic unit in $L_S \setminus \{\pm 1\}$ is bounded from below by $\nicefrac{(\log 2 + \log 3)}{2}=0.89587...$. We apply Lemma \ref{lem:1} with $y=15.9$ to deduce
\[
 h(\alpha)\geq \frac{\log 2}{15} = 0.04620\ldots
\]
for every $\al\in L_S^\times\setminus \{\pm 1\}$.
\end{example}

\begin{lemma}\label{Lambert}
 Let $a$ and $b$ be real numbers with $a >0$ and $b \geq 1+\log{a}$. Then $ax - b -\log{x}$ is positive for all real $x \geq (\nicefrac{8}{5}) a^{-1} (\log(a^{-1})+b)$.
\end{lemma}

\begin{proof}
 This is the second inequality of \cite[Lemma 3.3]{PottmeyerRamifi}. Note, that in the proof of this part of the lemma it is only required that $-\exp(-1) \leq -a\exp(-b)$ which is equivalent to our assumption $b \geq 1+\log{a}$.
\end{proof}

\begin{prop}\label{prop:nounit}
 Let $p$ be a rational prime number and $L_p/\bQ_p$ and $L_{\{p\}}/\bQ$ as usual in this paper. For every $\al\in L_{\{p\}}^\times$, which is not an algebraic unit, one has
\begin{equation}\label{eqn:prop1}
 h(\al)\geq \frac{\log(2)\log(p)}{5e(q+1) \log\left(\frac{5e(q+1)}{\log(p)}\right)}.
\end{equation}
\end{prop}

\begin{proof}
In order to apply Lemma \ref{lem:1} it is sufficient to find a real $y > p \geq p^{1/e}$ with
\begin{equation}\label{eqn1:prop1}
 -\frac{\log{y}}{2(y-1)}+\frac{1}{2} \left( (1-q^{-e})\frac{q\log{p}}{e(q^2 -1)}- \frac{\log{y}}{y}\right) - \frac{\log{2}}{y} \geq 0.
\end{equation}
This is equivalent to
\[
 \frac{q-q^{1-e}}{q^2 -1}\cdot\frac{\log p}{e}y - \log 4 -\frac{2y-1}{y-1}\log y \geq 0.
\]
Since $\nicefrac{2y-1}{y-1} \leq 3$ for $y \geq p$, and $\nicefrac{q-q^{1-e}}{q^2 -1}\geq \nicefrac{1}{q+1}$ this equation holds true if
\[
 \frac{\log p }{3e(q+1)}y - \frac{\log 4 }{3} - \log y \geq 0.
\]
The term $1+\log\left(\nicefrac{\log p}{3e(q+1)}\right)<1+\log(\nicefrac{1}{3})$ is always negative. In particular we can apply Lemma \ref{Lambert} to deduce that this last inequality, and hence \eqref{eqn1:prop1}, is satisfied for any $y$ satisfying
\begin{equation}\label{eqn2:prop1}
 y\geq \frac{24e(q+1)}{5\log p } \left(\log\left(\frac{3e(q+1)}{\log p }\right)+\frac{\log 4 }{3}\right) = \frac{24e(q+1)}{5\log p } \left(\log\left(\frac{3\sqrt[3]{4}e(q+1)}{\log p }\right)\right).
\end{equation}
We set $y=\frac{5e(q+1)}{\log p } \left(\log\left(\frac{5e(q+1)}{\log p }\right)\right)$, which obviously satisfies \eqref{eqn2:prop1}. The function $\nicefrac{\log\log x}{\log x}$ has its maximum at $\exp(-1)$. Hence,
\begin{align*}
y > 5\frac{p}{\log p} \log\left(5\frac{p}{\log p}\right) = 5p + 5p\frac{\log 5}{\log p} - 5p \frac{\log\log p}{\log p} > 5 p\left(1-\frac{\log\log p}{\log p}\right)  >p.
\end{align*}
Therefore, by Lemma \ref{lem:1} we can conclude that every $\al \in L_{\{p\}}^\times$ which is not an algebraic unit satisfies
\[
 h(\al)\geq \frac{\log 2}{y} \geq \frac{\log(2) \log(p)}{5e(q+1)\log\left(\frac{5e(q+1)}{\log p}\right)},
\]
proving the proposition.
\end{proof}

\begin{prop}\label{prop:unit}
 Let $p$ be a rational prime number and $L_p/\bQ_p$ and $L_{\{p\}}/\bQ$ as above. For every algebraic unit $\al\in L_{\{p\}}^\times$, which is not a root of unity, one has
\begin{equation}\label{eqn:prop2}
 h(\al)\geq \frac{\log p}{13e(q-1)\log\left(\frac{5e(q-1)}{\log p}\right)^4}.
\end{equation}
\end{prop}
\begin{proof}
First we note, that if $p$ is odd and $e=1$ then Proposition \ref{Petsche} gives the stronger bound $h(\alpha) \geq \frac{\log(p/2)}{q-1}$. For $p=2$ and $e\leq2$ the same proposition gives $h(\alpha)\geq \frac{\log 2}{4(q-1)}$ which is also a stronger bound than \eqref{eqn:prop2}. Hence, we can exclude this case as well.

Now the proof is almost the same as the proof of Proposition \ref{prop:nounit}. 
We want to apply Lemma \ref{lem:2}. To do so it is sufficient to find a real $z >\max\{7,p\}\geq \max\{7,p^{1/e}\}$ with
\begin{equation}\label{eqn1:prop2}
 - \frac{\log z}{2(z-1)}  + \frac{1}{2} \left((1 - q^{-e}) \frac{q\log{p}}{e(q -1)^2} - \frac{\log z}{z} \right) - \frac{1}{4z}\left(\frac{\log\log z}{\log z}\right)^3 \geq 0.
\end{equation}
This is equivalent to
\[
(1-q^{-e})\frac{q}{e(q-1)^2}\log(p)z -\frac{1}{2}\left(\frac{\log\log z}{\log z}\right)^3 -\frac{2z-1}{z-1}\log(z) \geq 0.
\]
In case $z >7$ we have $\nicefrac{2z-1}{z-1} \leq \nicefrac{13}{6}$. Using this, $(\nicefrac{\log\log z}{\log z} )^3 \leq \exp(-1)^3 < \nicefrac{1}{20}$, and $(1-q^{-e}) \nicefrac{q}{(q-1)^2}\geq \nicefrac{1}{q-1}$ we see that \eqref{eqn1:prop2} follows if
\[
 -\log z + \frac{6\log p }{13 e (q-1)}z-\frac{3}{260}>0.
\]
Hence, by Lemma \ref{Lambert}, this is satisfied for any $z> \max\{p,7\}$ with
\begin{equation}\label{eqn2:prop2}
 z\geq \frac{52}{15}\frac{e (q-1)}{\log p}\left( \log\left( \frac{13}{6}\frac{e (q-1)}{\log p}\right)+\frac{3}{260}\right).
\end{equation}
We claim that $ z=\nicefrac{7e (q-1)}{2\log p}\cdot  \log\left(\nicefrac{5e(q-1)}{ \log p}\right)$ is a valid choice. One easily checks that inequality \eqref{eqn2:prop2} is satisfied by this $z$. Moreover, one can calculate directly that $z > 7$ for all $p\leq 7$. The fact $z >p$, for $p>7$, follows exactly as in the proof of Proposition \ref{prop:nounit}.

Therefore, Lemma \ref{lem:2} tells us that for any algebraic unit $\al \in L_{\{p\}}$ which is not a root of unity we have
\begin{align}\label{eqn:unitheight}
 h(\al)&\geq \frac{\log p}{14e (q-1)\log\left(\frac{5e(q-1)}{\log p}\right)}\left(\frac{\log\log\left(\frac{7}{2}\frac{e (q-1)}{\log p} \log\left(\frac{5e(q-1)}{\log p}\right)\right)}{\log\left(\frac{7}{2}\frac{e(q-1)}{\log p} \log\left(\frac{5e(q-1)}{\log p}\right)\right)}\right)^3 \nonumber\\ &= \frac{\log p}{14e (q-1)\left(\log\left(\frac{5e(q-1)}{\log p}\right)\right)^4}\left(\frac{\log\log\left(\frac{7}{2}\frac{e (q-1)}{\log p} \log\left(\frac{5e(q-1)}{\log p}\right)\right)\log\left( \frac{5e(q-1)}{\log p} \right)}{\log\left(\frac{7}{2}\frac{e(q-1)}{\log p} \log\left(\frac{5e(q-1)}{\log p}\right)\right)}\right)^3.
\end{align}
The function 
\[
x \mapsto \frac{\log\log \left(\frac{7}{2}x\log(5x)\right) \log(5x)}{\log(\frac{7}{2}x\log(5 x))}
\]
is monotonically increasing in the interval $(1,\infty)$. As we assume that $e\geq 2$, and $e\geq3$ if $p=2$, the minimum of $\nicefrac{e(q-1)}{\log p}$ is attained for $e=2$, $f=1$ and $p=3$. It follows
\begin{align*}
 h(\al) &\geq \frac{\log(p)}{14e (q-1)\log\left(\frac{5e(q-1)}{\log p}\right)^4} \left(\frac{\log\log\left(\frac{14}{\log 3} \log\left(\frac{20}{\log 3}\right)\right) \log\left(\frac{20}{\log 3}\right)}{\log\left(\frac{14}{\log 3}\log\left(\frac{20}{\log 3}\right)\right)}\right)^3\\
&\geq \frac{\log p}{13e(q-1)\log\left(\frac{5e(q-1)}{\log p}\right)^4}
\end{align*}
for any algebraic unit $\al \in L_{\{p\}}$ which is not a root of unity.
\end{proof}

\begin{proof}[Proof of Theorem \ref{thm:2}]
In order to prove the first part of the theorem, we have to show that the lower bound \eqref{eqn:prop1} is always greater than the bound \eqref{eqn:prop2}.
This is equivalent to the statement
\[
\frac{\log\left(\frac{5e(q+1)}{\log p}\right)}{\log\left(\frac{5e(q-1)}{\log p}\right)^4} \leq \log(2)\frac{13(q-1)}{5(q+1)}.
\]
This inequality is true, since we have
\begin{align*}
& \frac{\log\left(\frac{5e(q+1)}{\log p}\right)}{\log\left(\frac{5e(q-1)}{\log p}\right)^4} = \frac{\log\left(\frac{5e(q-1)}{\log p}\right) + \log\left(\frac{q+1}{q-1}\right)}{\log\left(\frac{5e(q-1)}{\log p}\right)^4}  =  \frac{1 }{\log\left(\frac{5e(q-1)}{\log p}\right)^3} + \frac{\log\left(\frac{q+1}{q-1}\right)}{\log\left(\frac{5e(q-1)}{\log p}\right)^4} \\ \leq & \frac{1}{\log\left(\frac{5}{\log 2}\right)^3}+\frac{\log 3}{\log\left(\frac{5}{\log 2}\right)^4} \leq \log(2) \frac{13}{15} \leq \log(2) \frac{13(q-1)}{5(q+1)}.
\end{align*}
This proves the first part.

In case $e =1$ we can apply Proposition \ref{Petsche} to bound the height of algebraic units by $\nicefrac{\log(p/2)}{q-1}$ for an odd prime $p$ and $\nicefrac{\log 2}{2(2^f -1)}$ for $p=2$. As these estimates are always greater than \eqref{eqn:prop1}, the theorem follows.

\end{proof}

\bibliographystyle{abbrv} 
\bibliography{bib-quant}        

\end{document}